\documentclass[]{article}

\usepackage{lipsum}

\usepackage[centertags]{amsmath}
\usepackage{amsfonts}
\usepackage{amssymb}
\date{}
\usepackage{amsthm}
\usepackage{mathtools}
\usepackage{newlfont
}
\usepackage{amsmath,bm}
\newcommand{\be}{\begin{equation}}
\newcommand{\ee}{\end{equation}}

\theoremstyle{plain}
\newtheorem{thm}{Theorem}[section]
\newtheorem{cor}[thm]{Corollary}

\newtheorem{prop}[thm]{Proposition}
\newtheorem{conj}[thm]{Conjecture}
\theoremstyle{definition}

\theoremstyle{remark}

\begin{document}
\title{Group Testing with Pools of Fixed Size}
\author{David Cariolaro\footnote{Unfortunately, Professor David Cariolaro passed away before the completion of this paper.}
, Zhaiming Shen\footnote{E-mail address: zmshen2009@gmail.com}, Yi Zhang\footnote{ E-mail address: zhangyi000@ymail.com }
\\Department of Mathematical Sciences
\\Xi’an Jiaotong-Liverpool University
\\Suzhou, Jiangsu
\\
215123 CHINA}
\date{June 29, 2014}
\maketitle
\begin{abstract}
In the classical combinatorial (adaptive) group testing problem, one
is given two integers \(d\) and \(n\), where \(0\le d\le n\), and a population of \(n\)
items, exactly \(d\) of which are known to be defective. The question is to
devise an optimal sequential algorithm that, at each step, tests a subset of
the population and determines whether such subset is contaminated (i.e.
contains defective items) or otherwise. The problem is solved only when
the \(d\) defective items are identified. The minimum number of steps that
an optimal sequential algorithm takes in general (i.e. in the worst case)
to solve the problem is denoted by \(M(d, n)\). The computation of \(M(d, n)\)
appears to be very difficult and a general formula is known only for \(d = 1\).
We consider here a variant of the original problem, where the size of
the subsets to be tested is restricted to be a fixed positive integer \(k\). The
corresponding minimum number of tests by a sequential optimal algorithm
is denoted by \(M^{\lbrack k\rbrack}(d, n)\). In this paper we start the investigation of the
function \(M^{\lbrack k\rbrack}(d, n)\).
\end{abstract}
\section{Introduction}
Group Testing originated during World War II, in connection with the analysis
of blood samples [1]. Given a population of \(n\) items, each of which can be
either pure or defective, the problem is to determine all the defective items
in the population. Tests are applied to arbitrary subsets of the population,
and a positive result indicates the presence of a defective item in the subset
tested, whereas a negative result indicates the absence of defective items in the subset tested. Originally only probabilistic methods were used, but starting
with the paper of Li [3], no probability distribution assumptions were made
on the defective set and the combinatorial approach was introduced. In the
combinatorial version of the problem, it is generally assumed that the number \(d\)
of defective items is known in advance (although this may not be always true in
real situations). This is called the \((d, n)\)-problem in the monograph by Du and
Hwang [2], whose notation we follow here. The goal is to devise an algorithm
which will solve the problem using the minimum number of tests. Here the
implicit assumption is that the performance of an algorithm is measured in the
worst case, i.e. the algorithm must perform well not in the average, but under
worst-case scenario. We consider here only sequential algorithms, i.e. we assume
that tests are done in a sequence, and the result of any test is known before the
next test is performed. By an optimal algorithm we shall therefore designate
any sequential algorithm which solves the \((d, n)\)-problem using the minimum
number of tests in the worst case. Such number of tests is denoted by \(M (d, n)\).\\
One of the main challenges of Combinatorial Group Testing is the determination of the function \(M (d, n)\). This appears to be a very difficult problem, and
a complete answer (valid for every value of n) is known only for \(d = 1\).\\
We shall here consider a variant of the \((d, n)\)-problem. Specifically, we shall
assume that we are only allowed to test subsets of the population of size \(k\), where
\(k\) is a fixed positive integer. This assumption may seem quite arbitrary, but it
could be realistic in some situation and it gives rise to a new mathematical
challenge. Let \(M^{\lbrack k\rbrack}
(d, n)\) denote the number\footnote{If, for some particular values of \(k\), \(d\) and \(n\), the problem is unsolvable, we shall write
 \(M^{\lbrack k\rbrack}(d, n)=\infty\).} of tests (in the worst case) performed by an optimal sequential algorithm in order to solve the \((d, n)\)-problem,
with the restriction that each subset to be tested has size \(k\) (clearly \(0 \le k \le n\)).
In this paper we shall start the investigation of the function \(M^{\lbrack k\rbrack}
(d, n)\).

\section{Preliminary results}
A nonempty subset of the population is said to be \(pure\) if it consists only of
pure items, and \(contaminated\) otherwise. The set of defective items is called the
\(defective\) set. A \(t\)-subset is a subset of size \(t\). The complement of a subset \(X\)
of the population is denoted by \(\bar{X}\). The following facts are easily proved and sometimes useful.\\
\begin{prop}
\(M (1, n) =\lceil\log_2 n\rceil\).
\begin{proof}
In one direction the inequality is called the information-theoretic lower
bound and follows immediately from the consideration that the sample space
consists of \(n\) items and the tests have a binary outcome (either positive or
negative). In the other direction the inequality follows from an application of
the so-called halving method. This consists in testing, at each stage, a subset
which contains, roughly, half of the items than the previous subset tested. 
\end{proof}
\end{prop}

Notice however that the halving method can work if the first subset
tested contains no more than \(2^{\lceil \log_2 n\rceil-1}\) and no less than \(n-2^{\lceil \log_2 n\rceil-1}\) items.
\\
\begin{prop}
\(M^{\lbrack k \rbrack} (d, n) \ge M (d, n)\).
\begin{proof}
Indeed, when solving the \(\lbrace d, n\rbrace \)-problem, we can use subsets of any size,
and in particular of size \(k\).
\end{proof}
\end{prop}

\begin{prop}
\(M^{\lbrack k \rbrack} (0, n) = M^{\lbrack k \rbrack} (n, n)\).
\begin{proof}
Indeed in this case the knowledge of \(d\) directly yields the solution without
necessitating any test.
\end{proof}
\end{prop}
\begin{prop}
\(M^{\lbrack 0 \rbrack} (d, n) = M^{\lbrack n \rbrack} (d, n)=0\) for \(0 \le d \le n\).
\begin{proof}
Obvious, since no test is informative in this case.
\end{proof}
\end{prop}
In view of the above facts, henceforward we shall assume that \(0 < k < n\)
and \(0 < d < n\), without further mention.\\
\begin{prop}
\(M^{\lbrack 1 \rbrack} (d, n) = n-1\).
\begin{proof}
Indeed this is the individual testing, which requires all items except one
to be tested (the knowledge of \( d\) will render the \(n\)-th test unnecessary).
\end{proof}
\end{prop}
\begin{prop}
\(M^{\lbrack k \rbrack} (d, n) = \infty\) if \(k > n-d\).
\begin{proof}
Indeed any test done under the above conditions has a positive outcome,
and hence it is not informative.
\end{proof}
\end{prop}
\begin{prop}
\(M^{\lbrack n-d \rbrack} (d, n) = (^n_d)-1\).
\begin{proof}
There are exactly \((^n_d)\) different subsets of size \(n − d\). Exactly one of them is pure. Hence by testing all such subsets except one we can determine the pure set, and hence the set of defective items (which is just its complement). This proves
\begin{equation}
M^{\lbrack n-d \rbrack} (d, n)\le (^n_d)-1 
\end{equation}

To prove the reverse inequality assume that we have tested \((^n_d)-2\) subsets and
that they are all contaminated. Let \(P_1\), \(P_2\) be the two \((n- d)\)-subsets that have
not been tested. Notice that every set \(X\) which has been tested has a nonempty intersection with both \(\bar{P_1}\)
and \(\bar{P_2}\) (otherwise it would coincide with either \(P_1\) or
\(P_2\), contradicting the assumption). Thus at this stage we cannot decide whether the pure set equals \(P_1\) or \(P_2\), since both possibilities are consistent. Hence one further test is necessary, and this completes the proof of the proposition.
\end{proof}
\end{prop}
\section{The case \(d=1\)}
We now consider the case \(d = 1\). We have the following.
\begin{thm}\label{thm1}
\begin{equation}
M^{\lbrack k \rbrack} (1, n)\ge \lceil\frac{n}{k} \rceil -2+\lceil \log_2{(n-(\lceil\frac{n}{k}\rceil-2)\cdot k)} \rceil.
\end{equation}
Furthermore we have equality above if \(n \ge \max\lbrace 2k-2^{\lceil \log_2k\rceil -1 }, k +2^{\lceil \log_2k\rceil -2 } \rbrace\).
\begin{proof}
After testing \(\lceil \frac{n}{k} \rceil -2\) subsets, we are left with at least 
\begin{equation}
r = n - (\lceil\frac{n}{k} \rceil - 2) \cdot k
\end{equation}
items untested (notice that this number is between \(k + 1\) and \(2k\)). Assume that
all the tests so far are negative. Then we need to identify the defective item
among at least \(r\) items. This will take at least \(M (1, r) = \lceil \log_2 r\rceil \) tests. This proves that
\begin{equation}
M^{\lbrack k \rbrack} (1, n)\ge \lceil\frac{n}{k} \rceil -2+\lceil \log_2 r\rceil =\lceil\frac{n}{k} \rceil -2+\lceil \log_2{(n-(\lceil\frac{n}{k}\rceil-2)\cdot k)} \rceil.
\end{equation}
Now suppose that
\begin{equation}\label{1}
n \ge \max\lbrace 2k-2^{\lceil \log_2k\rceil -1 }, k +2^{\lceil \log_2k\rceil -2 } \rbrace.
\end{equation}
We shall describe an algorithm which solves the problem using at most \(\lceil\frac{n}{k} \rceil -2+\lceil \log_2{(n-(\lceil\frac{n}{k}\rceil-2)\cdot k)} \rceil\) tests. First test \(\lceil\frac{n}{k} \rceil -2\) disjoint subsets. Assume that one of them, say \(A\), tests positive. We claim that we can perform the halving method on \(A\), thereby completing the task with additional \(\lceil \log_2 k\rceil \)tests. Notice that this is within
the required bound, since, as already observed,
\begin{equation}
k\le n-(\lceil\frac{n}{k}\rceil -2)\cdot k,
\end{equation}
which implies
\begin{equation}
\lceil\frac{n}{k} \rceil -2+\lceil \log_2 k\rceil \le\lceil\frac{n}{k} \rceil -2+\lceil \log_2{(n-(\lceil\frac{n}{k}\rceil-2)\cdot k)} \rceil.
\end{equation}
To prove that we can perform the halving method on A, we notice (as we did after the proof of Proposition 1) that we can start by taking a subset of \(A\) of size \(2^t\), where \(t = \lceil \log_2 k\rceil - 1\). Call this subset \(B\). In order to be able to test \(B\),
we need to extend it to a set of size \(k\) by using additional pure items. In order to do this, we need to take \(k -2^t\) items which are not in \(A\). This is possible if
\begin{equation}
n \ge k - 2^t+ k = 2k - 2^t.
\end{equation}
This condition is guaranteed by our assumption \eqref{1}. At the next stage, we need to test a subset of \(B\) or \(A\setminus B\) of size \(2^{t-1}\)
(if such subset does not exist the test is unnecessary). Call this subset \(C\). In order to be able to test \(C\), we need to
extend \(C\) to a set of size \(k\) by introducing additional pure items. This will be possible if there exist \(k - 2^{t-1}\) items outside of \(B\), i.e. if
\begin{equation}\label{2}
n \ge k- 2^{t-1}+ 2^t= k + 2^{t-1}.
\end{equation}
This condition is guaranteed by our assumption \eqref{1}. We then need to test a subset \(D\) of \( 2^{t-2}\) items from \(C\) or \(B \setminus C\) or \((A \setminus B) \setminus C\), and, arguing as before,
we can do it only if we have enough items outside of \(C\) (or \(B \setminus C\) or \((A \setminus B) \setminus C\), respectively). The corresponding sufficient condition is
\begin{equation}
n \ge k − 2^{t-2}+ 2^{t-1}= k + 2^{t-2}
,
\end{equation}

which is certainly verified by \eqref{2}. In a similar way, we can then see that all further tests can be performed by adding a sufficient number of pure items to the set. We can then assume that all the initial \(\lceil \frac{n}{k}\rceil - 2 \) tests are negative. In this case we are left with exactly
\begin{equation}
r =n - (\lceil \frac{n}{k}\rceil - 2) \cdot k \le 2k
\end{equation}
items. Thus the halving method is again applicable in this case, and this shows that the number of additional tests required is \(\lceil\log_2 r\rceil\). This completes the proof. 
\end{proof}
\end{thm}
We notice the following corollary of Theorem \ref{thm1}.
\begin{cor}\label{cor1}
Let $n\geq 2k-1$. Then $M^{[k]}(1, n) = \lceil n/k \rceil-2+\lceil\log_2 (n-(\lceil n/k \rceil-2)\cdot k)\rceil$.
\end{cor}

\begin{proof}
Let $t=\lceil \log_2 k\rceil-1$. Clearly
\[2k-1\geq 2k-2^t\] and 
\[2k-1\geq k+2^{t-1}\]
since 
\[k>2^t\geq 2^{t-1}+1.\]
Hence the condition of second part of Theorem \ref{thm1} holds, and so does the conclusion.
\end{proof}

Theorem \ref{thm1} is stronger than Corollary \ref{cor1}, as for example, it implies that $M^{[4]}(1,6)=3$, whereas Corollary \ref{cor1} does not.
~\\
Notice that the condition of the second part of Theorem is tight as, for example, $M^{[4]}(1,5)=4\neq 3$. (The fact that $M^{[4]}(1,5)=4$ follows immediately from Proposition $7$.)

\section{The case k=2}

We now concentrate on the case $k=2$. The following is an immediate consequence of Theorem \ref{thm1} (or Corollary \ref{cor1}).

\begin{cor}
$M^{[2]}(1,n)=\lceil\frac{n}{2}\rceil$ for every $n\geq 3$.
\end{cor}

\begin{thm}
$$M^{[2]}(2,n)=\left\{
\begin{array}{rcl}
\infty && if {~n=3}\\
5 && if {~n=4}\\
\lfloor\frac{n}{2}\rfloor+2    &      & {if ~ n\geq 5}
\end{array} \right. 
$$
\end{thm}

\begin{proof}
The first two identities follow immediately from Proposition $6$ and Proposition $7$, respectively. Assume $n\geq 5$. We first show the inequality 
\begin{align} \label{inequality 1}
M^{[2]}(2,n)\geq \lfloor\frac{n}{2}\rfloor+2.
\end{align}

By Proposition $2$ and known results about function $M(d,n)$, we have 
\[M^{[2]}(2,5)\geq M(2,5)=4, ~M^{[2]}(2,6)\geq M(2,6)=5, ~M^{[2]}(2,7)\geq M(2,7)=5.\]

Hence \eqref{inequality 1} holds for $n=5,6,7$. Assume now $n\geq 8$. Suppose we perform $\lfloor\frac{n-5}{2}\rfloor$ tests and that the outcome is always negative. Since we are testing pairs, the maximum number of items that we have tested at this stage is $n-5$ if $n$ is odd and $n-6$ if $n$ is even, so that at least $5$ items remains to be tested if $n$ is odd and $6$ items remain to be tested if $n$ is even. Since  $M(2,5)=4$ and $M(2,6)=5$, we need in general at least $4$ more tests if $n$ is odd and $5$ more tests if $n$ is even to identify the defective item. Thus we conclude that

\[M^{[2]}(2,n)\geq\frac{n-5}{2}+4=\lfloor\frac{n}{2}\rfloor +2\] if $n$ is odd and
\[M^{[2]}(2,n)\geq\frac{n-6}{2}+5=\lfloor\frac{n}{2}\rfloor +2\] if $n$ is even. This proves the inequality \eqref{inequality 1}.

We now prove the reverse inequality for $n\geq 5$ by exhibiting a specific algorithm. 

The algorithm works as follows. First we test $\lfloor\frac{n}{2}\rfloor$ mutually disjoint pairs. 
~\\
\textbf{Case 1: $\bm{n}$ even.} Then all items have been tested. Assume only one of the pairs is contaminated. Then we conclude that such pair is the defective set. We can then assume that exactly two pairs are contaminated, say $\{a,b\}$ and $\{c,d\}$. Let $e$ be a pure item (which exists since by assumption $n\geq 5$, so that at least one pair has been identified as pure.) Using $e$ we can determind the two defective items in two further steps (namely, testing $\{a,e\}$ and testing $\{c,e\}$). Thus we have solved the problem in this case using $\lfloor\frac{n}{2}\rfloor+2$ tests, as required.

\textbf{Case 2: $\bm{n}$ odd}. Suppose first that the only one of the pairs tested, say $\{a,b\}$, is contaminated. Since $n\geq 5$, at least one pair has been tested and identified as pure. Choose an item $e$ from this pair. Test $\{a,e\}$ and $\{b,e\}$, thus determining exactly which items in $\{a,b\}$ are defective. If both $a$ and $b$ are defective, conclude that the defective set is $\{a,b\}$. If only one of 
$a$ and $b$, say $a$, is defective, conclude that $\{a,c\}$ is the defective set, where $c$ is the unique element that has not been tested. We have thus solved also this instance of the problem using $\lfloor\frac{n}{2}\rfloor+2$ tests and this completes the proof.

\end{proof}

For general $d$ we have the following upper bound on $M^{[2]}(d,n)$.

\begin{thm}\label{thm3}
$M^{[2]}(d,n)\leq \lceil n/2\rceil+2d-3$ if $3\leq d\leq \lfloor n/2\rfloor-1$.
\end{thm}

\begin{proof}
We describe an algorithm which solves the problem using at most $\lceil n/2\rceil+2d-3$ tests. Suppose first that $n$ is even. We first test $\lceil n/2\rceil=n/2$ disjoint pairs. Let $p$ be the number of contaminated pairs. Clearly $p\leq d$. If $p=d$, then necessarily some pairs will have tested negative, due to the assumption that $d\leq\lfloor n/2\rfloor-1$. It is also clear that there will be exactly one  defective item in each contaminated pair. Using a pure item, we can then determine the $d$ defective items with $d$ further tests. Thus in this case the problem can be solved in 
\[\lceil n/2\rceil+d\leq\lceil n/2\rceil+2d-3\]
tests.

Assume now that $p\leq d-1$. Also in this case some pairs will have tested negative, so that we have identified some pure items. Using a pure item, we can perform individual testing on the $2p$ items belonging to contaminated pairs. This will require $2p-1$ tests, as the last test is unnecessary. Thus in this case the problem may be solved with
\[\lceil n/2\rceil+2p-1\leq\lceil n/2\rceil+2d-3\]
tests, as required.

Suppose now that $n$ is odd. The proof for this case is identical, except that we start by testing $\lceil n/2\rceil-1=\frac{n-1}{2}$ disjoint pairs. One of these pairs will necessarily be pure since the number of defective items does not exceed $\lfloor n/2\rfloor-1=\frac{n-3}{2}$ by assumption. Using a pure item, we can then test the unique item which is so far untested. The remaining part of algorithm is exactly the same as for the other case.

\end{proof}

Intuitively it seems that the algorithm described in the proof of Theorem \ref{thm3} is an optimal algorithm, but we are unable to prove it formally. We leave it as a conjecture.

\begin{conj}
$M^{[2]}(d,n)=\lceil n/2\rceil+2d-3$ if $3\leq d\leq \lfloor n/2\rfloor-1$.
\end{conj}


\begin{thebibliography}{3}
\bibitem{R1}
R. Dorfman, \emph{The detection of defective members of large populations}, Ann.
Math. Statist., \(\bm{14}\) (1943), 436-440.
\bibitem{D1}
D.-Z. Du and F.K. Hwang, \emph{Combinatorial Group Testing and its Applications}, (2nd edition) World Scientific Publishing, Singapore, 2000.
\bibitem{L1}
 C.H.Li,\emph{ A sequential method for screening experimental variables}, J. Amer.
Statist. Assoc., \(\bm{57}\) (1962), 455-477.
\end{thebibliography}
\end{document}